\documentclass{amsart}

\usepackage[english]{babel}
\usepackage{bbm}
\usepackage{hyperref}
\usepackage{verbatim}

\newcommand{\N}{\ensuremath{\mathbb{N}}}
\newcommand{\R}{\ensuremath{\mathbb{R}}}

\newcommand{\E}{\ensuremath{\mathbb{E}}}
\renewcommand{\P}{\ensuremath{\mathbb{P}}}

\newcommand{\ind}[1]{\ensuremath{\mathbbm{1}_{\{#1\}}}}
\newcommand{\diff}{\mathop{}\mathopen{}\mathrm{d}}
\newcommand{\cal}[1]{\ensuremath{\mathcal{#1}}}
\newcommand{\Poi}{\cal{N}}

\def\etal{et al.}

\newtheorem{proposition}{Proposition}

\title[Transient Behavior of Ehrenfest and Engset Processes]{On the Transient Behavior of
  Ehrenfest and Engset Processes}
\date{\today}
\author[M. Feuillet]{Mathieu Feuillet}
\address[M. Feuillet,Ph. Robert]{INRIA Paris---Rocquencourt, Domaine de Voluceau, 78153 Le Chesnay, France}
\email{Mathieu.Feuillet@inria.fr}

\author[Ph. Robert]{Philippe Robert}
\email{Philippe.Robert@inria.fr}
\urladdr{\href{http://www-rocq.inria.fr/~robert}{http://www-rocq.inria.fr/\~{}robert}}

\begin{document}

\begin{abstract}
Two classical  stochastic processes are  considered, the Ehrenfest process,  introduced in
1907 in the kinetic  theory of gases to describe the heat  exchange between two bodies and
the  Engset  process,  one  of   the  early  (1918)  stochastic  models  of  communication
networks. This paper investigates the  asymptotic behavior of the distributions of hitting
times  of these  two processes  when  the number  of particles/sources  goes to  infinity.
Results concerning the hitting times of  boundaries in particular are obtained.  The paper
relies  on  martingale  methods,  a  key  ingredient is  an  important  family  of  simple
non-negative  martingales, an  analogue, for  the  Ehrenfest process,  of the  exponential
martingales used in the study of random walks or of Brownian motion.

\vspace{5mm} \bigskip\centerline{\em \`A la m\'emoire de Philippe Flajolet.}

\end{abstract}

\maketitle

\hrule

\vspace{-3mm}

\tableofcontents

\vspace{-8mm}

\hrule

\bigskip

\section{Introduction}

\subsection*{The Ehrenfest Process}
In  this paper  one  considers the  following  continuous time  version  of the  classical
Ehrenfest  urn model. This process has been introduced  to study the  heat exchange  between
bodies.  One assumes that each particle of a set of $N$ particles is located in one of two
boxes (bodies),  $0$ and  $1$ say. A  particle in box  $0$ [resp.  $1$] goes into  box $1$
[resp. $0$] at rate $\nu$ [resp. $\mu$]. One denotes as $E_N(t)$ the number of particles in
box~$1$ at time $t\geq 0$. This birth and death process can also be represented as
$E_N(t)=Y_1(t)+\cdots+Y_N(t)$, where $(Y_i(t), 1\leq i\leq N)$ are $N$ i.i.d. Markov jump
processes with values in $\{0,1\}$. 

Originally, the model is a discrete time process $(Z_N(k))$ and each unit of time a particle
is taken at random (i.e. equally likely) to  be moved from one box to the other, $Z_N(k)$ is
the number of particles in box $1$  at time~$k$.  This corresponds to the symmetrical case
$\mu=\nu$  and,  clearly,  $E_N(t)$  can  be  represented  as  $Z_N(\Poi_{N\mu}((0,t]))$  if
  $\Poi_{N\mu}$ is a Poisson process with rate $N\mu$.  The process $(E_N(t))$ follows the
  same  path as $(Z_N(k))$  but on  a time  scale with  a factor  $N\mu$.  The  Markov chain
  $(Z_N(k))$ is also random walk on the graph of the {\em hypercube} $\{0,1\}^N$ where edges
  connect  elements that  differ at  only one  coordinate. The  equilibrium  properties of
  $(E_N(t))$ and $(Z_N(k))$  are fairly well known, in particular a  quite precise estimate of
  the   duration  of  time   to  reach   equilibrium  is   available.   See   Diaconis  et
  al.~\cite{Diaconis}.

Results on transient  quantities of this process, like the distribution  of the first time
when the  box $0$ is empty,  are more rare. There  are generic results on  birth and death
process which describe some of these distributions in terms of spectral characteristics of
the associated  infinitesimal generator: the spectral  measure and a  family of orthogonal
polynomials.       See     Karlin      and      McGregor~\cite{Karlin:02,Karlin-2}     and
Keilson~\cite{Keilson:01}.   In practice,  the corresponding  orthogonal  polynomials and,
sometimes, their  spectral measure do not  have a simple  representation, this complicates
significantly detailed  investigations of  these hitting times.   In the  symmetrical case
$\mu{=}\nu$,  Bingham~\cite{Bingham}  and  Flajolet  and  Huillet~\cite{Flajolet}  give  a
representation of their distributions.  In the general case Crescenzo~\cite{Crescenzo} and
Flegg et  al.~\cite{Flegg} provide expressions  of the densities.   It turns out  that the
expressions obtained in these papers involve,  in general, sums of combinatorial terms for
which asymptotic  results (when $N$  goes to infinity)  may be difficult to  obtain.  Note
that   this   is   nevertheless  done   in   the   symmetrical   case  in   Flajolet   and
Huillet~\cite{Flajolet}.

\subsection*{The Engset Process}
This   is  one  of   the  oldest   stochastic  models   of  communication   networks.  See
Engset~\cite{Engset}.  For this model  there are  $N$ sources  of communication  which are
active (state  $1$) or  inactive (state $0$).  An active  source becomes inactive  at rate
$\mu$. The total  number of simultaneous active sources cannot  exceed the quantity $C_N$,
the number of circuits of the network. An inactive source can therefore become active only
if there are  already strictly less $C_N$ active  sources, in this case it  occurs at rate
$\nu$. If $X_N(t)$  is the number of  active communications at time $t$,  when $C_N=N$ the
process  is just  the  Ehrenfest process.   Otherwise  $(X_N(t))$ can  be  described as  a
reflected version  of $(E_N(t))$. In  this context, an  important quantity is the  time it
takes to have  the full capacity of the  network used, i.e.  the hitting time  of $C_N$ by
$(X_N(t))$.  If the  equilibrium probability that $X_N$ is $C_N$ is  known, under the name
of {\em  Engset Formula}. To  the best of  our knowledge, results concerning  this hitting
time are quite rare, in particular for possible asymptotics when $N$ goes to infinity.

\subsection*{A Storage System}
Another, more recent, motivation for  considering $(X_N(t))$ is the stochastic analysis of
a storage  system where  files are duplicated  on $C_N$  servers. Each server  breaks down
independently at rate $\mu$ in which case it is repaired but all its files are lost.  As a
simplified model,  $X_N(t)$ is defined  as the  number of copies  of a specified  file, if
$X_N(t)=x$ then  a copy of the file  is lost if one  of the servers breaks  down, i.e.  at
rate $x\mu$.   If $X_N(t)=0$, there is no  copy of the file  in the system, it  is lost so
that $0$ is an absorbing point. If $0<x<C_N$ then  a new copy of the file may be added but
at rate  $N-x$, $N$ is the  maximal capacity of duplication  of the system.   It is easily
seen that  as long  as $(X_N(t))$  does not hit  $0$, $(X_N(t))$  is precisely  the Engset
process.  In this context it is of special interest to study the distribution of the first
time when the file is lost, i.e. the hitting time of $0$. See Feuillet and
Robert~\cite{Feuillet:02}. 

\subsection*{A Collection of Exponential Martingales}
This  paper  relies  heavily  on  the  use of  martingales  to  derive  explicit,  simple,
expressions of the Laplace  transforms of the hitting times of a  state of the system. One
obtains expressions  of these transforms as  ratios of simple integrals  for which various
asymptotic results, when  $N$ goes to infinity, can be derived  quite easily with standard
technical tools. In particular  one does not need to cope with  the asymptotic behavior of
sums of  combinatorial expressions.   Quite surprisingly, up  to now, martingales  did not
play  a major  role in  the previous  studies of  the Ehrenfest  process. One  can mention
Simatos and  Tibi~\cite{Simatos} where a martingale  approach is used  to estimate certain
exit times  for multi-dimensional Ehrenfest  processes. It is  one of the results  of this
paper to show  that a simple and  important family of martingales allows  a quite detailed
investigation of this process, and also of its variants like the Engset process.

The  key  ingredient  of  this  paper  is  a set of non-negative  martingales which will
be called exponential martingales.  If $(M(t))$ is a martingale on some probability space,
the associated {\em exponential   martingale} is the solution $(Z(t))$ of the stochastic
differential equation (SDE) 
\[
\diff Z(t)=Z(t-)\diff M(t), \quad t\geq 0,
\]
where $Y(t-)$ is the left  limit of $Y$ at $t$ and $\diff Y(t)$ is  the limit on the right
of $t$  of $s\mapsto Y(s)-Y(t-)$.  It is called  the {\em  Dol\'eans exponential} of  $(M(t))$. See
Chapter~IV  of Rogers  and Williams~\cite{rogers-00}  for  example.  Despite  there is  an
exponential martingale  for each martingale, a  small subset of these  martingales plays an
important role. For  the standard Brownian motion  $(B(t))$ this is the martingale
\[
\left(\exp\left(\beta B(t)-{\beta^2}t/2 \right)\right),
\]
for a fixed $\beta\in\R$. It is very helpful to derive the explicit expressions of Laplace
transforms   of  hitting   times   associated   to  Brownian   motion.    See  Revuz   and
Yor~\cite{Revuz:01}.  For  jump processes,  this is less  clear. It  does not seem  that a
``classification'' of exponential martingales exists  in general, even for birth and death
processes.   See Feinsilver~\cite{Feinsilver}  for  related questions.   Some examples  of
important processes are reviewed.

For  $\xi\in\R_+$, $\Poi_\xi$ denotes a Poisson process with rate $\xi$ and
$(\Poi_{\xi,i})$ a sequence of i.i.d. such Poisson processes.  All Poisson
processes are assumed to be independent. 
\begin{itemize}
\item {\bf Random Walks.} The classical exponential martingale associated to the random
  walk $(S(t))=(\Poi_\lambda([0,t])-\Poi_\mu([0,t]))$ is given by, for $\beta\in\R$,
\begin{equation}
\left(\exp\left[-\beta S(t) -t \left(\lambda
\left(1-e^{-\beta}\right)+\mu \left(1-e^{\beta}\right)\right)\right]\right).
\end{equation}
It is the exponential martingale associated to the martingale
\[
\left(\rule{0mm}{4mm}\beta(S(t)-(\lambda-\mu)t)\right).
\]
The corresponding reflected process is the $M/M/1$ queue with input rate $\lambda$ and
service rate $\mu$. 
\item {\bf The $\mathbf{M/M/\infty}$ Process.} This is a classical Markov process on $\N$ whose
  $Q$-matrix $Q=(q(x,y))$ is, for $x\in\N$, $q(x,x+1)=\lambda$ and $q^N(x,x-1)=
  \mu x$. 
It can be also seen as a kind of discrete Ornstein-Uhlenbeck
process, defined as the solution $(L(t))$ of the following SDE
\[
\diff L(t)={\cal N}_\lambda(\diff t)-\sum_{i=1}^{L(t-)} {\cal N}_{\mu,i}(\diff t).
\]
The following martingale has been introduced in Fricker
\etal~\cite{Fricker:07}, for $\beta\in\R$,  
\begin{equation}\label{PC}
\left(\left(1+\beta e^{\mu t}\right)^{L(t)} \exp\left(-\beta e^{\mu t}{\lambda}/{\mu}\right)\right). 
\end{equation}
It is the exponential martingale associated to the martingale
\[
\left(\int_0^t \left(1{+}\beta e^{\mu s}\right)[\Poi_\lambda(\diff s){-}\lambda\diff s]
-\sum_{i=1}^{+\infty}\int_0^t \frac{1}{1{+}\beta e^{\mu s}}\ind{i<L(s-)}[\Poi_{\mu,i}(\diff s){-}\mu\diff s]\right).
\]
\item {\bf The Ehrenfest Process}. Such a process $(E_N(t))$ with $N$ particles can be seen as the solution of
  the SDE
\[
\diff L(t)=\sum_{i=1}^{N-L(t-)} {\cal N}_{\nu,i}(\diff t)-\sum_{i=1}^{L(t-)} {\cal N}_{\mu,i}(\diff t),
\]
as it will be seen the corresponding exponential martingale is given by, for $\beta\in\R$,
\begin{equation}
\left(\left(1-\beta\mu e^{(\mu+\nu)t}\right)^{E_N(t)} \left(1+\beta\nu e^{(\mu+\nu)t}\right)^{N-E_N(t)}\right).
\end{equation}
It is the exponential martingale associated to the martingale defined by, up to the
multiplicative factor $\beta(\mu{+}\nu)$, 
\[
\left(
\sum_{i=1}^N  \int_0^t  e^{(\mu{+}\nu) s}\left(\ind{Y_i(s-)=1}[\Poi_{\mu,i}(\diff s){-}\mu\diff s]
{-}\ind{Y_i(s-)=0}[\Poi_{\nu,i}(\diff s){-}\nu\diff s]\right)
\right),
\]
where  $(Y_i(t))$ are  such  that $E_N(t)=Y_1(t)+\cdots+Y_N(t)$.  Recall  that the  Engset
process is a reflected version of this process.
\end{itemize}
From  these exponential  martingales, explicit  expressions of  Laplace transforms  of the
distribution of hitting times associated to these processes can be derived.  It may be not
be as straightforward  as in the case of  Brownian motion since the space  variable $t$ is
not separated  from the space variable, but  a convenient integration with  respect to the
free parameter $\beta$ solves the  problem. See Chapter~5 and~6 of Robert~\cite{Robert:08}
for the  $M/M/1$ and  $M/M/\infty$ processes and  Section~\ref{martsec} for  the Ehrenfest
process.

From the point of view of potential theory, these martingales are associated to the set of
extreme harmonic  functions. This  statement can  be made precise  in terms  of space-time
Martin  boundary.   See  Lamperti  and  Snell~\cite{Lamperti:06}  and  the  discussion  in
Section~\ref{martsec}.

\subsection*{Organization of the paper}
In  Section~\ref{secmodel},  the  two  stochastic  processes  are  defined  precisely.  In
Section~\ref{martsec}, the exponential martingale for the Ehrenfest process is introduced,
and, based on it, several interesting martingales for the Ehrenfest process and the Engset
process are constructed. As a corollary,  closed form expressions of the Laplace transform
of the hitting time of a given state  are obtained as the ratio of simple integrals.  This
holds in particular for the blocking time for the Engset process.  The last three sections
are devoted to the analysis of the  asymptotic behavior of the distribution of the hitting
time of $C_N$ and  $0$ when $N$ goes to infinity in such a  way that $C_N{\sim}\eta N$ for
some  $0<\eta\leq 1$.  Each  section consider  one  of the  three  possible regimes:  {\em
  sub-critical} when  $\nu<\eta$, the  process ``lives'' in  the neighborhood of  $\nu N$,
{\em super-critical} when the difference  $C_N-X_N(t)$ converges to a finite process, and,
finally, {\em critical} when $C_N-X_N(t)$ is of the order of $\sqrt{N}$.  For each regime,
by  taking advantage of  the simple  expressions of  the corresponding  Laplace transforms
obtained, various convergence in distribution results are derived.


\section{The Stochastic Model}\label{secmodel}
\subsection*{The Ehrenfest process}
Let $(Y(t))$ be  the simple Markov process  on $\{0,1\}$  whose $Q$-matrix $Q_Y$ is given
by 
\[
Q_Y=\begin{pmatrix}
-\nu&\nu\\\mu&-\mu
\end{pmatrix}.
\]
For $N\in\N$, 
if $(Y_i(t))$, $1\leq i\leq N$, are $N$ independent copies of $(Y(t))$, the Ehrenfest process $(E_N(t))$ is 
also a birth and death process but on the state space $\{0,1,\ldots,N\}$, it is 
defined as 
\begin{equation}\label{engdef}
E_N(t)=Y_1(t)+ Y_2(t)+\cdots+Y_N(t).
\end{equation}
The $Q$-matrix of $(E_N(t))$ will be denoted as $Q_{E_N}$, for $x\in\{0,\ldots,N\}$,
\begin{equation}\label{QmatEhr}
q_{E_N}(x,x-1) = \mu x \text{ and } 
q_{E_N}(x,x+1) = \nu (N-x).
\end{equation}

\subsection*{The Engset process}
For $1\leq  C_N\leq N$,  the Engset  process $(X_N(t))$ is  a birth  and death  process on
$\{1,\ldots,C_N\}$  which can be seen as a reflected version of $(E_N(t))$ at the
boundary $C_N$, i.e. its  $Q$-matrix $Q_{X_N}=(q_N(x,y))$  is given by, for $0\leq x\leq
C_N$,
\begin{equation}\label{Qmat}
q_{X_N}(x,x-1) = \mu x \text{ and } 
q_{X_N}(x,x+1) = \nu (N-x) \text{ if } x<C_N.
\end{equation}
In particular  the process  $(X_N(t))$ has the  same distribution  as the
process $(E_N(t))$  constrained to  the state space  $\{0,\ldots, C_N\}$. In particular, when $C_N=N$,
the two processes $(X_N(t))$ and $(E_N(t))$  starting from the same initial state have the
same distribution.

As  ergodic birth and death processes, the Markov processes $(X_N(t))$ and $(E_N(t))$ are
reversible and their stationary distribution at $x$ is, up to a normalization constant,
given by  
\[
\binom{N}{x}\left(\frac{\nu}{\mu}\right)^x,
\]
if $x$ is an element of  their respective state space.

\subsection*{Normalization of the time scale}
By considering the time scale $t\to t/(\nu+\mu)$ in the analysis of the processes $(E_N(t))$ and
$(X_N(t))$, it can be assumed without any loss of generality that $\nu+\mu=1$. This will
be the case in this paper.  

\subsection*{A limiting regime}
In the following it will be assumed that the constant  $C_N$ is asymptotically of the
order of $N$, i.e. that
\begin{equation}\label{rho}
\eta\stackrel{\text{def.}}{=}\lim_{N\to+\infty}{C_N}/{N},
\end{equation}
holds for some $\eta\in(0,1]$.

For $t$ large the probability that the variable $Y(t)$ defined above is at $1$ is given by
its equilibrium distribution at $1$, that is $\nu$.  The law of large numbers gives that $E_N(t)$
is of the  order of $N\nu$. Roughly  speaking, if $N\nu<C_N$ for $N$  large, i.e. $\nu<\eta$,
then the boundary at $C_N$ should not  play a significant role for first order quantities
related to $(X_N(t))$ and therefore the processes  $(X_N(t))$ and $(E_N(t))$ should have the same
behavior in the limit. On the contrary if $\nu\geq \eta$, due to the reflecting boundary at
$C_N$ for $(X_N(t))$, the Ehrenfest and the Engset processes should behave differently. 
This phenomenon will be stated more precisely in the last three sections of the paper. 

\section{Positive Martingales}\label{martsec}
Several families of positive martingales for the Ehrenfest and the Engset processes are
introduced in this section.  More specifically, when $(Z(t))$ is either $(X_N(t))$ or
$(E_N(t))$, one identifies a set of functions  $f:\N\times\R_+\mapsto \R_+$ such that the
process $(f(Z(t),t)$ is a martingale, i.e. that, for $t\geq 0$, the relation 
\[
\E\left(f(Z(t),t)\mid {\cal F}_t\right)=f(Z(s),s), \text{ for } s\leq t,
\]
holds almost surely, where ${\cal F}_t)$is the natural filtration associated to $(Z(t))$. 

If $Q_Z=(q_Z(\cdot,\cdot))$ is the $Q$-matrix of $(Z(t))$, this probabilistic property is
equivalent to the fact that the function $f$ is {\em space-time harmonic} with respect to $Q_Z$,
i.e. that the relation
\begin{equation}\label{sptime}
\frac{\partial}{\partial t} f(x,t)+ Q_Z(f(\cdot,t))(x)=0
\end{equation}
holds for $x\in\N$ and $t\geq 0$, where for $h:\N\mapsto\R_+$,
\[
Q_Z(h)(x)=\sum_{y\in\N} q(x,y)h(y).
\]
A space-time harmonic function of the Markov process $(Z(t))$ is just an harmonic function
of the  transient Markov process $((Z(t),t))$.  See  Appendix~B of Robert~\cite{Robert:08}
for  example.   When $(Z(t))$  is  $(E_N(t))$,  one will  prove  that  there  is a  family
$f_\beta$, $\beta\in\R$ of  such functions.  As it will be seen,  these martingales can be
interpreted  as  exponential  martingales.  They  will  give  in  particular  an  explicit
expression for  the Laplace transform  of the hitting  times associated to  both processes
$(E_N(t))$ and $(X_N(t))$.

For a given birth  and death process, there is already a  complete description of all such
positive martingales.   This is the  (space-time) Martin boundary  of the birth  and death
process, see  Lamperti and  Snell~\cite{Lamperti:06}. This description  is, unfortunately,
expressed in terms of the orthogonal polynomials associated to the birth and death process
which are defined, in general, by induction, see Karlin and Mc~Gregor~\cite{Karlin:02}. As
long as moments of some  transient characteristics are investigated, these martingales can
be used  but they  are not,  in general, really  helpful to  analyze the  distributions of
hitting times.

This situation is quite  classical, for Brownian motion for example, for  which there is a
family of  martingales indexed by $N\in\N$, if  $H_N$ is the Hermite  polynomial of degree
$N$,  then  $(M_N(t))=(t^{N/2}H_N(B(t)/\sqrt{t}))$ is  a  martingale.   Another family  of
martingales  is provided  by  the exponential  martingale $(\exp(\beta  B(t)-\beta^2t/2))$
indexed by $\beta\in\R$. This exponential martingale can be expressed as a weighted sum of
the  martingales $(M_N(t))$,  but  to get  explicit  expressions of  the distributions  of
hitting times, it  is the really useful martingale. See  Revuz and Yor~\cite{Revuz:01} for
example.  In the  case  of birth  and death  processes,  a general  result concerning  the
construction  of such  exponential  martingales  from the  martingales  associated to  the
orthogonal polynomials does not seem to exist.

\subsection{Exponential Martingales for the Ehrenfest Process}
Due  to the  simple  structure of  the  Ehrenfest process,  these  martingales are  really
elementary.  Nevertheless they play  a fundamental  role, most  of the  asymptotic results
obtained in this paper are based on these martingales. A more general version in a
multi-dimensional context has been introduced by Simatos and Tibi~\cite{Simatos}.
\begin{proposition}[Exponential martingales] \label{prop:hbeta}
For  $\beta\in\R$, the process 
\begin{equation}\label{expmart}
(M_{N}^{\beta}(t)) =  \left(\left(1-\beta\mu e^{t}\right)^{E_N(t)}
\left(1+\beta\nu e^{t}\right)^{N-E_N(t)}\right)
\end{equation}
is a martingale.
\end{proposition}
\begin{proof}
Define, for $N\geq 1$, $t\geq 0$ and $0\leq x\leq N$,
\begin{equation}\label{hn}
h_N^\beta(x,t)=(1-\beta\mu e^t)^x(1+\beta\nu e^t)^{N-x}.
\end{equation}
Clearly, the relations
\[
 \frac{\partial}{\partial t} h_1^\beta(0,t){=}\beta\nu e^t{=}
{-}Q_Y(h_1^\beta(\cdot,t))(0)\;\text{ and }\;
\frac{\partial}{\partial t} h_1^\beta(1,t){=}{-}\beta\mu e^t{=} {-}Q_Y(h_1^\beta(\cdot,t))(1),
\]
hold, where $Q_Y$ is the $Q$-matrix of $(Y(t))$ introduced in
Section~\ref{secmodel}. Consequently, the function $h_1^\beta$ is space-time harmonic for the
matrix $Q_Y$, equivalently $(f(Y(t),t))$ is a martingale. 

If $((Y_i(t)), 1\leq i\leq N)$ are $N$ i.i.d. processes with the same distribution as
$(Y(t))$, then, by using the independence of the processes $(Y_i(t))$, $i=1,\ldots,N$ and
Equation~\eqref{engdef}, one gets that the process 
\[
\left(\prod_{i=1}^N f(Y_i(t),t)\right)\stackrel{\text{dist.}}{=}\left(M_N^\beta(t)\right)
=\left(h_N^\beta(E_N(t),t)\right)
\]
is also martingale with respect to the filtration $({\cal F}_t){=}(\sigma{<}Y_i(s), s\leq t,
1\leq i\leq N{>})$. In particular the function $h_N^\beta$ is space-time harmonic with respect
to $Q_{E_N}$. The proposition is proved. 
\end{proof}
\subsection*{Martingales Associated to Orthogonal Polynomials}
As remarked by Karlin and McGregor~\cite{Karlin:02,Karlin}, the polynomials associated to the
Ehrenfest process are the $N+1$ {\em Krawtchouk polynomials} $({\cal K}_n^N,0\leq n\leq
N)$ defined by 
\begin{equation}\label{Kraw}
{\cal K}_n^N(x)=\binom{N}{n}^{-1}\sum_{\ell=0}^n (-1)^\ell
\binom{x}{\ell}\binom{N-x}{n-\ell}\left(\frac{\mu}{\nu}\right)^\ell,\quad  0\leq
n,x\leq N.
\end{equation}
These polynomials are orthogonal with respect to the binomial distribution
$$\left(\binom{N}{k}\nu^k\mu^{N-k},0\leq k\leq N\right).$$  The 
classical identity, see Karlin and McGregor~\cite{Karlin:02} for example, 
\begin{equation}\label{Gen}
\sum_{\ell=0}^N
\binom{N}{\ell}{\cal K}_n^N(x)u^\ell=(1+u)^{N-x}\left(1-\frac{\mu}{\nu}u\right)^x, \quad u\in\R,
\end{equation}
and the above proposition give that 
\[
\left(\sum_{n=0}^N \binom{N}{\ell}{\cal K}_n^N(E_N(t))\beta^n e^{n t}\right)=\left(M_N^{\beta/\nu}(t)\right)
\]
is a martingale. As a consequence one gets that, for any $0\leq n\leq N$, the process
$\left({\cal K}_n^N(E_N(t))\exp(nt)\right)$
is a martingale. The martingale $(M_N^\beta(t))$ can  thus be seen as an encoding of these
$(N+1)$ martingales in  the same way as the exponential martingale  of the Brownian motion
with  the Hermite  polynomials,  or the  martingale~\eqref{PC}  with the  Poisson-Charlier
polynomials. See Robert~\cite{Robert:08}.

Note that  the space  variable $E_N(t)$  and the time  variable $t$  are not  separated in
Expression~\eqref{expmart} of  the exponential martingale.  Provided that it can  be used,
Doob's optional stopping theorem applied to  some hitting time of some specified state $x$
does not  give useful information  on the distribution  of this variable.  But  given that
there  is  a  free  parameter  $\beta\in\R$ in  Expression~\eqref{expmart}  and  that  the
martingale property is  clearly preserved by integration with respect  to $\beta$, one may
try to find a  measure on $\R_+$ that will ``separate'' the  space and time variables. The
following proposition uses such a method.
\begin{proposition}\label{prop:IJalpha}
For any $\alpha>0$ and $t\geq 0$,  if
\begin{equation} \label{eq:IJ}
\begin{cases}
\displaystyle I_\alpha^N(t) = e^{-\alpha t} \int_0^1 (1-u)^{E_N(t)} \left(1+\frac{\nu}{\mu}u\right)^{N-E_N(t)}\hspace{-4mm} u^{\alpha-1}\,\diff u,\\
\displaystyle J_\alpha^N(t)  = e^{-\alpha t} \int_0^1 (1-u)^{N-E_N(t)} \left(1+\frac{\mu}{\nu}u\right)^{E_N(t)} u^{\alpha-1}\,\diff u,
\end{cases}
\end{equation}
and 
\[
T_x^{E_N} = \inf\{ t> 0: E_N(t) =x \}, \quad 0\leq x\leq N,
\]
then 
$(I_\alpha(t\wedge T_0^{E_N}))$ and $(J_\alpha(t \wedge T_N^{E_N}))$ are martingales.
\end{proposition}
\begin{proof}
Since $(N-E_N(t))$ is also an Ehrenfest process but with the two parameters $\mu$ and $\nu$
exchanged, one needs only to prove that the process $(I_\alpha(t\wedge T_0^{E_N}))$ is a martingale. 

Define
\[
f_N(x,t) \stackrel{\text{def.}}{=} \int_0^{e^{-t}/\mu} h_N^\beta(x,t)\beta^{\alpha-1}\diff \beta,
\]
where $h_N^\beta$ is defined by Equation~\eqref{hn}, then
\[
\frac{\partial f_N}{\partial t}(x,t)= 
\int_0^{e^{-t}/\mu} \frac{\partial h_N^\beta}{\partial t}(x,t)\beta^{\alpha-1}\diff \beta
-\frac{e^{-t}}{\mu}h_N^{e^{-t}/\mu}(x,t).
\]
Note that the last term of the above expression is $0$ if $x\not=0$. Consequently, for
$x\not=0$ and $t\geq 0$,
\[
\frac{\partial f_N}{\partial t}(x,t)+Q_{E_N}(f_N)(x,t)=
\int_0^{e^{-t}/\mu} \left[\frac{\partial h_N^\beta}{\partial t}(x,t)+Q_{E_N}(h_N^\beta)(x,t)\right]\beta^{\alpha-1}\diff \beta=0,
\]
because $h_N^\beta$ is space-time harmonic with respect to $Q_{E_N}$ as it has been seen
in the proof of Proposition~\ref{prop:hbeta}. In other words, the
function $f_N$ is space-time harmonic for the $Q$-matrix of the stopped process $(E_N(t\wedge T_0^{E_N}))$, hence 
\[
\left(f_N\left(E_N\left(t\wedge T_0^{E_N}\right),t\wedge T_0^{E_N}\right)\right)=\left(I_\alpha^N\left(t\wedge T_0^{E_N}\right)\right)
\] is a
martingale. The proposition is proved.
\end{proof}
It is now easy to get a representation of the Laplace transform of the hitting times for
the Ehrenfest process.
\begin{proposition}[Laplace Transform of Hitting Times]\label{prop:laplace}
For $0\leq x \leq y \leq N$ and if  $T_x^{E_N} = \inf\{ t> 0: E_N(t) =x \}$, 
the relations
\begin{equation}\label{eq:laplace-c}
\E_y\left( e^{-\alpha T_x^{E_N}} \right) = \frac{B_x(\alpha)}{B_y(\alpha)},
\quad\text{ and }\quad
\E_x\left( e^{-\alpha T_y^{E_N}} \right) = \frac{D_x(\alpha)}{D_y(\alpha)}
\end{equation}
hold, with
\begin{equation}\label{eq:BD}
\begin{cases}
\displaystyle B_x(\alpha) = \int_0^1 (1-u)^x \left(1+\frac{\nu}{\mu}u\right)^{N-x} \hspace{-3mm}u^{\alpha-1}\diff u,\\
\displaystyle D_x(\alpha) = \int_0^1 (1-u)^{N-x} \left(1+\frac{\mu}{\nu}u\right)^x u^{\alpha-1}\diff u.
\end{cases}
\end{equation}
\end{proposition}
There is in fact only one result here since $(N-E_N(t))$ is, as it has already been remarked, 
an Ehrenfest process with the parameters $\mu$ and $\nu$ exchanged. The second
relation of~\eqref{eq:laplace-c} is therefore a consequence of the first one. 
\begin{proof}
The martingale $I_\alpha(t\wedge T_0)$ is bounded
and then uniformly integrable. Therefore, Doob's optional stopping Theorem gives the relation
\[
\E_y(I_\alpha(0)) =  \E_y(I_\alpha(T_x)),
\]
the first relation of~\eqref{eq:laplace-c} follows.
\end{proof}
By expanding one of the terms of the integrand of Equation~\eqref{eq:BD}, ones gets
\begin{align*}
B_x(\alpha)&=\sum_{k=0}^{N-x}\binom{N-x}{k}\left(\frac{\nu}{\mu}\right)^k
\int_0^1 (1-u)^x u^{\alpha+k-1}\,\diff u\\
&=\sum_{k=0}^{N-x}\binom{N-x}{k}\left(\frac{\nu}{\mu}\right)^k
\frac{\Gamma(x+1)\Gamma(\alpha+k)}{\Gamma(\alpha+x+k+1)},
\end{align*}
by  using the  relation between  Beta functions  and Gamma  functions.  See  Whittaker and
Watson~\cite{Whittaker:01} for example. Laplace  transforms of hitting times can therefore
also expressed as  ratio of such sums, it  is quite common for hitting times  of birth and
death  processes. See  Equation~(4.4) of  Karlin and  McGregor~\cite{Karlin}  for example.
Flajolet and Huillet~\cite{Flajolet} uses this  kind of representations in the symmetrical
case.  As it  will be seen, from the  compact representation~\eqref{eq:BD} with integrals,
one will get asymptotic results for the distribution of these variables with standard
techniques. 
\subsection{Martingales for the Engset Process}
It has been seen that the Engset  process $(X_N(t))$ is a reflected version of the process
$(E_N(t))$   at   the   boundary   $C_N$.     The   two   families   of   martingales   of
Proposition~\ref{prop:IJalpha} cannot  be used directly  if the sample path  of $(E_N(t))$
may exceed $C_N$,  when the hitting times of  $0$ is analyzed for example. The  idea is to
construct a linear combination of the martingales $(I_\alpha(t))$ and $(J_\alpha(t))$ such
that the  space-time harmonicity  of the  corresponding function which  is valid  when the
space variable in $\{1,\ldots, C_N-1\}$ holds also at the boundary $C_N$.  This method has
been used in Kennedy~\cite{Kennedy:02} in the case of reflected random walks.
\begin{proposition}\label{propmartref}
For $\alpha>0$, define
\begin{equation}
\begin{cases}
\displaystyle b_{N}(\alpha) = \nu\int_0^1 (1-u)^{C_N}\left(1+\frac{\nu}{\mu}u\right)^{N-{C_N}-1}\hspace{-6mm}u^{\alpha}\,\diff u,\\
\displaystyle d_{N}(\alpha) = \mu\int_0^1 (1-u)^{N-{C_N}-1}\left(1+\frac{\mu}{\nu}u\right)^{{C_N}}u^{\alpha}\,\diff u,
\end{cases}
\end{equation}
and 
\[
(K^N_\alpha(t))= (d_{N}(\alpha)I^N_\alpha(t)+b_{N}(\alpha)J^N_\alpha(t)),
\]
where $(I^N_\alpha(t))$ and $(J^N_\alpha(t))$ are defined by Equations~\eqref{eq:IJ} with
$E_N(t)$ replaced by $X_N(t)$ then,  if
$T_0^{X_N}$ is the hitting time of $0$ by $(X_N(t))$,  the process $(K^N_\alpha(t\wedge T_0^{X_N}))$ is a martingale. 
\end{proposition}
\begin{proof}
Define the function $g_{\alpha}$ such that, for all $t\geq 0$ and $0\leq x \leq {C_N}$,
\begin{multline*}
g_{\alpha}(x,t) =  d_{N}(\alpha)\, e^{-\alpha t} \int_0^1(1-u)^x\left(1+\frac{\nu}{\mu}u\right)^{N-x}\hspace{-5mm}u^{\alpha-1}\,\diff u\\
+ b_{N}(\alpha)\,e^{-\alpha t} \int_0^1(1-u)^{N-x}\left(1+\frac{\mu}{\nu}u\right)^x u^{\alpha-1}\,\diff u.
\end{multline*}
The function $g_{\alpha}$ is space-time harmonic for the matrix $Q_{X_N}$ on $\{1,\ldots,
C_N-1\}$, that is 
\[
\left[\frac{\partial g_{\alpha}}{\partial t} + Q_{X_N}(g_{\alpha})\right](x,t) = 0, \quad 0< x<{C_N},
\]
since the two matrices $Q_{X_N}$ and $Q_{E_N}$ are identical as long as the starting point
is in $\{1,\ldots, C_N-1\}$  and that $(I^N_\alpha(t\wedge T_0^{E_N}))$ and $(J^N_\alpha(t\wedge
T_N^{E_N}))$ are martingales by Proposition~\ref{prop:IJalpha}. 

The space-time harmonicity of $g_{\alpha}$ for the matrix $Q_{E_N}$ at  $C_N<N$ gives the relation
\[
\left[\frac{\partial g_{\alpha}}{\partial t} + Q_{X_N}(g_{\alpha})\right]({C_N},t) = -\nu(N-{C_N})\left[g_{\alpha}({C_N}+1,t)-g_{\alpha}({C_N},t)\right].
\]
For $0\leq y\leq 1$, one has 
\begin{multline*}
(1{-}y)^{{C_N}+1}\left(1{+}\frac{\nu}{\mu}y\right)^{N{-}C_N{-}1}\hspace{-12mm}{-}(1{-}y)^{{C_N}}\left(1{+}\frac{\nu}{\mu}y\right)^{N-{C_N}}\hspace{-8mm}=
{-}\frac{y}{\mu} (1{-}y)^{{C_N}}\left(1{+}\frac{\nu}{\mu}y\right)^{N{-}C_N{-}1},
\end{multline*}
and, with a similar identity, one gets the relation
\[
\left[\frac{\partial g_{\alpha}}{\partial t} + Q_{X_N}(g_{\alpha})\right](C_N,t) = 
\frac{(N-{C_N})}{\mu}\left(d_N(\alpha)b_N(\alpha)-b_N(\alpha)d_N(\alpha)\right)=0.
\]
The function $g_{\alpha}$ is space-time harmonic for the $Q$-matrix of the stopped process 
$(X_N(t\wedge T_0^{X_N}))$, the process $(K^N_\alpha(t\wedge T_0^{X_N}))$ is therefore a martingale. 
\end{proof}

\begin{proposition}[Laplace Transform of Hitting Times for Engset Process]\label{prop:laplace-bis}
For $0\leq x\leq y\leq C_N$, if 
$T_x^{X_N}=\inf\{s\geq 0: X_N(s)=x\}$,
then, for $\alpha\geq 0$, 
\begin{equation}\label{eq:laplace-0-bis}
\E_x\left( e^{-\alpha T_y^{X_N}} \right) = \frac{D_x(\alpha)}{D_y(\alpha)} \text{ and }\E_y\left(e^{-\alpha T_x^{X_N}}\right) = \frac{d_{N}(\alpha)B_y(\alpha)+b_{N}(\alpha)D_y(\alpha)}{d_{N}(\alpha)B_x(\alpha)+b_{N}(\alpha)D_x(\alpha)},
\end{equation}
with the notations of Propositions~\ref{prop:laplace} and~\ref{propmartref}.
\end{proposition}
\begin{proof}
The first identity comes from the fact that the two processes $(X_N(t))$ and $(E_N(t))$
starting from the same initial state are identical in distribution as long as they do not
reach $C_N$. In particular, if $X_N(0)=E_N(0)=x$, the variables $T_y^{E_N}$ and $T_{y}^{X_N}$
have the same distribution. The second identity is a direct consequence of the martingale property of
$(K^N_\alpha(t\wedge T_0^{X_N}))$ proved in the above proposition. 
\end{proof}

\section{A Fluid Picture}
This section gives a quick description of the first order properties of the
Ehrenfest and Engset processes as $N$ goes to infinity. Its purpose is mainly to introduce
the three natural possible asymptotic regimes that will be investigated in detail in the
last sections.  The proofs of the asymptotic results are quite standard and therefore will be skipped. 

From now on, it is assumed that Relation~\eqref{rho} holds, that is
\[
\lim_{N\to+\infty}{C_N}/{N}=\eta>0.
\]
The Engset process $(X_N(t))$ can also be seen as the unique solution of the following stochastic
differential equation
\[
\diff X_N(t)=\ind{X_N(t-)<C_N} \sum_{i=1}^{N-X_N(t-)} \Poi_{\nu,i}(\diff t)-\sum_{i=1}^{X_N(t-)}
\Poi_{\mu,i}(\diff t),
\]
starting from $X_N(0)$, where, for $\xi>0$, $(\Poi_{\xi,k})$ denotes an i.i.d. sequence of
Poisson processes with rate $\xi$. The different Poisson processes used are independent.  

The initial state is assumed to satisfy
\[
\lim_{N\to+\infty}{X_N(0)}/{N}=x_0\in [0,\eta],
\]
then, by complementing Poisson processes in order to get martingales, the above equation
can be rewritten as 
\[
\diff X_N(t)= \diff M_N(t)+ [\nu(N-X_N(t))\ind{X_N(t)<C_N}-\mu X_N(t)]\diff t,
\]
where $(M_N(t))$ is a martingale of the order of $\sqrt{N}$. In the same way as for the Erlang
process, see Chapter~6 of Robert~\cite{Robert:08} for example, one can prove the following
convergence in distribution of processes
\[
\lim_{N\to+\infty} \left({X_N(t)}/{N}\right)=\left(\min(\eta, \nu+(x_0-\nu)e^{-t})\right).
\]
This first order description of the Engset process shows that there are three different
asymptotic regimes.
\begin{itemize}
\item {\em Super-Critical Regime:} $\nu>\eta$.\\
Under this condition the renormalized process is at the boundary $C_N$ at time
\begin{equation}\label{tstar}
t^*\stackrel{\text{def.}}{=}\log\left({(\nu-x_0)}/{(\nu-\eta)}\right).
\end{equation}
A more detailed picture can  be obtained by looking at the process 
\[
(Z_N(t))=(C_N-X_N(t/N))
\]
of empty spaces with a ``slow'' time scale. As $N$ goes to infinity, is is easily seen
that the $Q$-matrix of this birth and death process converges to the $Q$-matrix of an
ergodic $M/M/1$ process with input rate $\eta$ and service rate $\nu$. In particular, this
gives the asymptotic expression of the Engset formula, for $t\in\R_+$,
\[
\lim_{N\to+\infty} \P(X_N(t)=C_N)= 1-{\eta}/{\nu}.
\]
\item {\em Sub-Critical Regime:} $\nu<\eta$.\\
In this case, one has in fact
\[
\lim_{N\to+\infty} \left({X_N(t)}/{N}\right)
=\left(\nu+(x_0-\nu)e^{-t}\right)=\lim_{N\to+\infty} \left({E_N(t)}/{N}\right).
\]
As expected, the boundary at $C_N$ does not play a role, for the first order the Engset
process and the Ehrenfest process are identical.
\item {\em Critical Regime:} $\nu=\eta$.\\
The fluid limit picture gives that the system saturates ``at infinity'' which is a too
rough description of its evolution as it will be seen. 
\end{itemize}

The next  sections are devoted  to the asymptotic  analysis of the distributions  of hitting
times.   For the  sake of  simplicity, it  is assumed  that the  initial state  is  on the
boundary, either  $0$ or $C_N$. Similar  results could be obtained  without any additional
difficulty when the  initial state is in the neighborhood of  some $\lfloor z N\rfloor$ for
$0\leq z\leq \eta$.

\section{Super-Critical Regime}\label{supsec}
As it has been seen, under the condition $\nu>\eta$ and at time $t^*$ defined by Equation~\eqref{tstar},
the system is saturated for the fluid limit. It implies in particular that the hitting
time of the boundary $C_N$,
\[
T_{C_N}^{X_N}=\inf\{s\geq 0; X_N(s)=C_N\}
\]
converges in distribution to $t^*$. The following proposition gives a more precise
asymptotic result. See Theorem~3 of Flajolet and Huillet~\cite{Flajolet} for a related result in the
symmetrical case. 
\begin{proposition}\label{supprop}
If $C_N=\eta N+O(1)$, $\eta<\nu$ and $X_N(0)=0$, then the sequence of random variables
\[
\left(\sqrt{N}\left[T_{C_N}^{X_N}-\log\left({\nu}/{(\nu-\eta)}\right)\rule{0mm}{4mm}\right]\right)
\]
converges in distribution to a centered normal random variable with variance
\[
{\sqrt{\eta(1-\eta)}}/{(\nu-\eta)}
\]
\end{proposition}
\begin{proof}
Proposition~\ref{prop:laplace} gives the equation
\begin{multline}\label{JQ}
\E_0\left(e^{-\alpha \sqrt{N} T_{C_N}^{X_N}}\right) = 
\int_0^{\sqrt{N}} \left(1-\frac{u}{\sqrt{N}}\right)^N u^{\sqrt{N}\alpha-1}\diff u\\
\left/\int_0^{\sqrt{N}} \left(1-\frac{u}{\sqrt{N}}\right)^{N-C_N} \left(1+\frac{\mu}{\nu} \frac{u}{\sqrt{N}}\right)^{C_N} u^{\sqrt{N}\alpha-1}\diff u,\right.
\end{multline}
for $\alpha\geq 0$. 
The integrand of the  numerator of the right-hand side  of the above equation can be expressed as $\exp(f_N(u))$ with  $$f_N(u)=N\log\left(1-{u}/{\sqrt{N}}\right) + \left(\sqrt{N}\alpha-1\right)\log u.$$ 
The function has a unique maximum at 
\[
y_N=\frac{\alpha\sqrt{N}-1}{\sqrt{N}+\alpha-1/\sqrt{N}}= \alpha-\frac{1+\alpha^2}{\sqrt{N}}+o\left(1/\sqrt{N}\right),
\]
and 
\[
\begin{cases}
\displaystyle f_N(y_N)= (\alpha\log(\alpha)-\alpha)\sqrt{N}-\frac{\alpha^2}{2}-\log(\alpha)+o(1),\\
\displaystyle f_N^{''}(y_N)= -\frac{(1+\alpha^2)}{\alpha^2}-\frac{\sqrt{N}}{\alpha}+o(1).
\end{cases}
\]
Laplace's method, see Flajolet and Sedgewick~\cite{Flajolet-09} for example,  gives therefore the relation
\begin{align}\label{speq1} 
\int_0^{\sqrt{N}} e^{f_N(u)}\diff u&\sim \frac{\sqrt{2\pi}}{\sqrt{-f^{''}(y_N)}}e^{f_N(y_N)}
\\& \sim \frac{\sqrt{2\pi\alpha}}{N^{1/4}}\exp\left((\alpha\log(\alpha)-\alpha)\sqrt{N}-\frac{\alpha^2}{2}-\log(\alpha)\right).\notag
\end{align}
Similarly, the integrand of the denominator of the right-hand side of Equation~\eqref{JQ} is
$\exp(g_N(u))$ with 
\[
g_N(u)=(N-C_N)\log\left(1-\frac{u}{\sqrt{N}}\right) + C_N\log\left(1+\frac{\mu}{\nu}\frac{u}{\sqrt{N}}\right)+\left(\sqrt{N}\alpha-1\right)\log u.
\]
This concave function on the interval  $(0,\sqrt{N})$ has  a unique maximum located at $z_N=z_0-\delta/\sqrt{N}+o(1/\sqrt{N})$, with
\[
z_0=\alpha \frac{\nu}{\nu-\eta} \text{ and } \delta 
=\nu{\frac {  {\alpha}^{2}\left( {\nu}^{2}+\eta-2\eta\nu\right)+(\nu-\eta)^2}{ \left( \nu-\eta \right) ^{3}}},
\]
and, with some calculations,  one gets the following  expansions
\begin{multline*}
g_N(z_N)=\left(-\alpha+\alpha\log(\alpha)+\alpha\log\left(\frac{\nu}{\nu-\eta}\right)\right)\sqrt{N}
\\+\frac {\left(2\nu\eta-{\nu}^{2}-\eta\right)}{ \left(\nu-\eta \right) ^{2}}\frac{\alpha^2}{2}-\log(\alpha)-\log\left(\frac{\nu}{\nu-\eta}\right)+o(1)
\end{multline*}
and
\[
g_N^{''}(z_N)=-\left(1-\eta+\eta\frac{(1-\nu)^2}{\nu^2}\right)-\frac{\alpha\sqrt{N}-1}{z_0^2}+o(1).
\]
By using again Laplace's method, this gives the relation
\begin{multline}\label{speq2} 
\exp\left(-\alpha \sqrt{N}\log\frac{\nu}{\nu-\eta}\right)\int_0^{\sqrt{N}} e^{g_N(u)}\,\diff u\\\sim 
\frac{\sqrt{2\pi\alpha}}{N^{1/4}}\exp\left((\alpha\log(\alpha)-\alpha)\sqrt{N}+\frac {\left(2\nu\eta-{\nu}^{2}-\eta\right)}{ \left(\nu-\eta \right) ^{2}}\frac{\alpha^2}{2}-\log(\alpha)\right).
\end{multline}
Equation~\eqref{JQ} together with Relations~\eqref{speq1} and~\eqref{speq2} give finally
\[
\lim_{N\to+\infty} \E\left(\exp\left(-\alpha\sqrt{N}\left[T_{C_N}^{X_N}-\log({\nu}/{(\nu-\eta)})\right]\right)\right)=
\exp\left(\frac{\eta(1-\eta)}{(\nu-\eta)^2}\frac{\alpha^2}{2}\right),
\]
the proposition is proved. 
\end{proof}

\subsection*{An Informal Proof}
The limit theorem obtained in Proposition~\ref{supprop} is a consequence of some detailed,
annoying, but simple, calculations used to apply Laplace method. One can get quite quickly
an idea of the possible limit with the help of the exponential martingale $(M_N^\beta(t))$
of Proposition~\ref{prop:hbeta} through a non-rigorous derivation.  As it will be seen, it
gives the  correct answer but  its justification seems  to be difficult. The  main problem
comes from the fact that, in this martingale, the term $e^t$ stopped at some random time may not be integrable
at  all.  For  example,  it  is  easily  seen  that  the  first  jump  of  the  martingale
$(M_1^\beta(t))$ is not a regular stopping time for this martingale, i.e. the optional
stopping stopping theorem is not valid for this stopping time. 

Denote  $Z_N=\sqrt{N}\left(\exp(T^{X_N}_{C_N})-\exp(t^*)\right)$, 
where $t^*$ is, as before, $\log(\nu/(\nu-\eta))$.
By using the martingale~\eqref{expmart} of Proposition~\ref{prop:hbeta} by assuming that
the stopping time $T^{X_N}_{C_N}$ is regular for it, one gets
\[
\E\left[\left(1-\mu e^{T^{X_N}_{C_N}}{\beta}/{\sqrt{N}}\right)^{C_N}
\left(1+\nu e^{T^{X_N}_{C_N}}{\beta}/{\sqrt{N}}\right)^{N-C_N}\right]=
\left(1+\nu{\beta}/{\sqrt{N}} \right)^{N},
\]
this equation can be written as $\E\left(\exp(U_N)\right)=1$, 
with
\begin{multline*}
U_N=
{C_N}\log\left(1-\mu e^{T^{X_N}_{C_N}}{\beta}/{\sqrt{N}}\right)\\+
(N-C_N)\log\left(1+\nu e^{T^{X_N}_{C_N}}{\beta}/{\sqrt{N}}\right)
-N\log\left(1+\nu{\beta}/{\sqrt{N}} \right),
\end{multline*}
hence,
\begin{multline*}
U_N= -\beta (\eta-\nu) Z_N-\left(\left(\eta\mu^2+(1-\eta)\nu^2\right)e^{2t^*}-\nu^2\right){\beta^2}/{2}
+o\left({1}/{N}\right)\\=-\beta (\eta-\nu) Z_N-\frac{\nu^2\eta(1-\eta)}{(\nu-\eta)^2}\beta^2/2+o\left({1}/{N}\right),
\end{multline*}
provided that the limit can be taken under the integral, one gets  finally 
\[
\lim_{N\to+\infty}\E\left(e^{-\beta Z_N}\right)=
\exp\left(\frac{\beta^2}{2}\frac{\nu^2\eta(1-\eta)}{(\nu-\eta)^4}\right).
\]
Expressed as a limit theorem for $T^{X_N}_{C_N}$, this is precisely the above proposition.

\section{Sub-Critical Regime}\label{subsec}
It is assumed in this section that $\nu<\eta$ so that the Ehrenfest process ``lives'' in the
interior of the state space, the hitting time of the boundary $C_N$ should be therefore
quite large. The following propositions give asymptotic results concerning this phenomenon.

The first result concerns the time it  takes to the Ehrenfest process to have all particles
in  one box  when,  initially, they  are all  in the  other
box. This  is of course  a very natural  quantity for this  process. In the  discrete time
case,  representations  of   the  average  of  this  quantity   have  been  obtained in a
symmetrical setting.  See Bingham~\cite{Bingham} and references therein.
\begin{proposition}\label{propsub1}
If $\nu<1$, $X_N(0)=0$ and $C_N=N$, then the sequence of random variables
$\left(N \nu^N T_{C_N}^{X_N}\right)$ converges in distribution to an
exponentially distributed random variable with parameter $1-\nu$. 
\end{proposition}
\begin{proof}
One uses Equation~\eqref{eq:laplace-c} of Proposition~\ref{prop:laplace} to get that, 
for $N\geq 1$, 
\begin{equation}\label{eqaux}
\E\left(e^{-\alpha_N T_{N}^{X_N}}\right)= 
\int_0^1(1-u)^Nu^{\alpha_N-1}\,\diff u\left/\int_0^1\left(1+\frac{\mu}{\nu}u\right)^Nu^{\alpha_N-1}\,\diff u\right.,
\end{equation}
with $\alpha_N=\alpha N \nu^N$, for some  $\alpha>0$. 

The numerator of this expression can be written as, after an integration by parts, 
\begin{multline*}
\int_0^1(1-u)^Nu^{\alpha_N-1}\,\diff u=\int_0^1N(1-u)^{N-1} \frac{u^{\alpha_N}}{\alpha_N}\,\diff u
\\=N^{-\alpha_N}\int_0^N \left(1-\frac{u}{N}\right)^{N-1}
\frac{u^{\alpha_N}}{\alpha_N}\,\diff u\sim \frac{1}{\alpha_N},
\end{multline*}
by Lebesgue's Theorem. 

By subtracting $1/\alpha_N$ from the denominator of the right-hand side of
Equation~\eqref{eqaux}, one gets 
\begin{align*}
\Delta_N&\stackrel{\text{def.}}{=}\int_0^1\left(\left(1+\frac{\mu}{\nu}u\right)^N-1\right)u^{\alpha_N-1}\,\diff
u=\frac{N\mu}{\nu}\int_0^1
\left(1+\frac{\mu}{\nu}u\right)^{N-1}\frac{1-u^{\alpha_N}}{\alpha_N}\,\diff u\\
&=\frac{1}{\nu^N}\int_0^N \mu\left(1-\mu \frac{u}{N}\right)^{N-1}\frac{1-(1-u/N)^{\alpha_N}}{\alpha_N}\,\diff u,
\end{align*}
hence,
\begin{align*}
\alpha_N\Delta_N
&=\frac{\alpha_N}{N\nu^N}\int_0^N \mu\left(1-\mu \frac{u}{N}\right)^{N-1}\frac{1-(1-u/N)^{\alpha_N}}{\alpha_N/N}\,\diff u
\\&\sim \alpha \int_0^{+\infty}\mu e^{-\mu u} u\,\diff u=\frac{\alpha}{\mu}. 
\end{align*}
These two asymptotic results plugged into Equation~\eqref{eqaux} give the desired
convergence in distribution. 

\end{proof}
Theorem~2  of Bingham~\cite{Bingham}  provides a  similar result  in the  symmetrical case
$\mu=\nu$ and in discrete time. In the  present case, there is an additional factor $N$ in
the scaling of $T_{C_N}^{X_N}$ which is due  to the fact that the continuous time dynamics are
$N$ times faster than the discrete time case.

\begin{proposition}\label{propsub2}
If $C_N=\eta N+O(1)$, $\nu<\eta<1$ and $X_N(0)=0$, then if
\[
H=(1-\eta)\log\left(\frac{1-\eta}{1-\nu}\right)+\eta\log\left(\frac{\eta}{\nu}\right), 
\]
the sequence of random variables
\begin{equation}\label{hitrare}
\left(\frac{\sqrt{\eta(1-\eta)}}{(\eta-\nu)\sqrt{2\pi}}\sqrt{N}e^{-N H} \,T_{C_N}^{X_N}\right),
\end{equation}
converges in distribution to an exponentially distributed random variable with parameter
$1$.
\end{proposition}
One remarks that  the exponential decay factor $H$  of the above proposition is  in fact a
relative  entropy of  Bernoulli random  variables  with respective  parameters $\eta$  and
$\nu$. Despite similar ``entropy'' expressions  appeared at several occasions in the study
of these processes, we have not been able to find a simple explanation for the occurrences
of these constants.
\begin{proof}
For $\alpha>0$, denote by $\alpha_N$ the product of $\alpha$ and the coefficient of
$T_{C_N}^{X_N}$ in Expression~\eqref{hitrare}.
Equation~\eqref{eq:laplace-c} of Proposition~\ref{prop:laplace} is again used  
\begin{equation}\label{eqaux2}
\E\left(e^{-\alpha_N T_{N}^{X_N}}\right){=} \!\!
\int_0^1(1{-}u)^Nu^{\alpha_N{-}1}\,\diff u\left/\!\!\int_0^1(1{-}u)^{N-C_N}\left(1{+}\frac{\mu}{\nu}u\right)^{C_N}u^{\alpha_N{-}1}\,\diff u\right.,
\end{equation}
The asymptotic behavior of the numerator of this Laplace transform has already been obtained in the proof of the
above proposition. 

To study the denominator, one will proceed as before. For $u\in[0,1]$, denote 
\[
f_N(u)=(N-C_N)\log(1-u)+C_N\log\left(1+{\mu}u/{\nu}\right),
\]
this function has a unique maximum at 
\[
y_0\stackrel{\text{def.}}{=}\frac{C_N/N-\nu}{1-\nu}=\frac{(\eta-\nu)}{(1-\nu)}+O(1/N)
\]
which is given by
\[
f_N(y_0)=
\left[(1-\eta)\log\left(\frac{1-\eta}{1-\nu}\right)+\eta\log\left(\frac{\eta}{\nu}\right)\right]
N +o(1),
\]
and
\[
f_N''(y_0)=-\frac{(1-\nu)^2}{\eta(1-\eta)}N+o(1).
\]
The denominator of Laplace Transform~\eqref{eqaux2} is
\begin{multline*}
\int_0^1\left[(1{-}u)^{N-C_N}\left(1{+}\frac{\mu}{\nu}u\right)^{C_N}-1\right]u^{\alpha_N{-}1}\,\diff
u=\int_0^1\left[e^{f_N(u)}-1\right]u^{\alpha_N{-}1}\,\diff u
\\=
\int_0^1 f_n'(u)e^{f_N(u)}\frac{(y_0^{\alpha_N}-u^{\alpha_N})}{\alpha_M}\,\diff u+\frac{1-y_0^{\alpha_N}}{\alpha_N},
\end{multline*}
by integration by parts. The integral $I_N$ of the right-hand side of this relation can be written as 
\[
I_N=\int_{-y_0\sqrt{N} }^{(1-y_0)\sqrt{N}}\frac{1}{\sqrt{N}}f_N'\left(y_0+\frac{u}{\sqrt{N}}\right)e^{f_N(y_0+u/\sqrt{N})}
\frac{(y_0^{\alpha_N}-(y_0+u/\sqrt{N})^{\alpha_N})}{\alpha_N}\diff u,
\]
hence,
\begin{align*}
I_N &= y_0^{\alpha_N-1}\frac{-f_N''(y_0)}{N^{3/2}} e^{f_N(y_0)} \int_{-\infty}^{+\infty}
u^2\exp\left(\frac{f_N''(y_0)}{N}\frac{u^2}{2}\right)\,\diff u+o(1/N)\\
& =\frac{1}{y_0}\sqrt{\frac{2\pi}{-f_N''(0)}} e^{f_N(y_0)} +o(1/N).
\end{align*}
By gathering these asymptotic results one gets the relation
\[
\lim_{N\to+\infty} \E\left(\exp\left(-\alpha_N T_{N}^{X_N}\right)\right)={1}/{(1+\alpha)},
\]
the proposition is proved. 

\end{proof}

\begin{proposition}[Hitting time of the empty state]\label{prop:empty}
Under the condition  $\nu<\eta$ and if $C_N=\eta N+O(1)$ for $\eta>0$ and $X_N(0)=C_N$, then the sequence of variables
\[
\left(N(1-\nu)^NT_{0}^{X_N}\right)
\]
converges in distribution to an exponential random variable with parameter $\nu$.
\end{proposition}
Not   that   this   result   can   be,   informally,  justified   with   the   result   of
Proposition~\ref{propsub1}. Without  the boundary  $C_N$, one could  obtain the  result by
exchanging $\mu$ and $\nu$ and by  using Proposition~\ref{propsub1}. This result shows in particular
that the boundary does not change the limiting behavior of $T_0^{X_N}$ in the sub-critical regime.
\begin{proof}
Denote $\alpha_N=N(1-\nu)^N$. Proposition~\ref{prop:laplace-bis} gives the relation
\begin{equation}\label{eq11}
\E_{C_N}\left(e^{-\alpha_N T_0^{X_N}}\right) = \frac{d_{N}(\alpha_N)B_{C_N}(\alpha_N)+b_{N}(\alpha_N)D_{C_N}(\alpha_N)}{d_{N}(\alpha)B_0(\alpha_N)+b_{N}(\alpha_N)D_0(\alpha_N)}.
\end{equation}
One starts with the asymptotic behavior  of $(d_N(\alpha_N))$,  
\begin{align*}
d_N&(\alpha_N)=\mu\int_0^1 \left(1-u\right)^{N-C_N-1}\left(1+\frac{\mu}{\nu}u\right)^{C_N}u^{\alpha_N}\,\diff u\\
&=  \frac{\mu}{\sqrt{N}}\int_0^{\sqrt{N}} \left(1-\frac{u}{\sqrt{N}}\right)^{N-C_N-1}\left(1+\frac{\mu}{\nu}\frac{u}{\sqrt{N}}\right)^{C_N}\,\diff u+o\left(1/\sqrt{N}\right)\\
&=
\frac{1}{\sqrt{N}}\exp\left(\frac{\eta-\nu}{\nu}\sqrt{N}\right)\int_0^{+\infty}\exp\left(-\left(\frac{\eta}{\nu^2}+\frac{(1-\eta)}{(1-\nu)^2}\right)\frac{u^2}{2}\right)\,\diff
u+o\left(1/\sqrt{N}\right).
\end{align*}
The other coefficient $b_{N}(\alpha_N)$ is such that 
\begin{align*}
b_{N}(\alpha_N) &= \nu\int_0^1 (1-u)^{C_N}\left(1+\frac{\nu}{\mu}u\right)^{N-C_N}u^{\alpha_N}\,\diff u\\ 
&= \frac{\nu}{N}\int_0^N\left(1-\frac{u}{N}\right)^{C_N}\left(1+\frac{\nu}{\mu N}u\right)^{N-C_N-1}\,\diff u+o\left(1/{N}\right)\\ 
&= \frac{\nu(1-\nu)}{\nu-\eta}\frac{1}{N}+o\left(1/{N}\right).
\end{align*}

The proof of Proposition~\ref{propsub1} provides the following relations
\[
D_0(\alpha_N)=\int_0^1(1-u)^Nu^{\alpha_N-1}\,\diff u\sim \frac{1}{\alpha_N},
\]
and 
\[
B_0(\alpha_N)=\int_0^1\left(1+\frac{\nu}{\mu}u\right)^Nu^{\alpha_N-1}\,\diff u
\sim\frac{1}{\alpha_N} + \frac{1}{\nu N(1-\nu)^N}.
\]
The two remaining terms to estimate  are
\begin{align*}
B_{C_N}(\alpha_N)&=\int_0^1(1-u)^{C_N}\left(1+\frac{\nu}{\mu}u\right)^{N-C_N}u^{\alpha_N-1}\,\diff u,\\
D_{C_N}(\alpha_N)&=\int_0^1(1-u)^{N-C_N} \left(1+\frac{\mu}{\nu}u\right)^{C_N}u^{\alpha_N-1}\,\diff u.
\end{align*}
With the same method as in the proof of
Proposition~\ref{propsub2}, one can show that the quantities
$B_{C_N}(\alpha_N)$ and  
$D_{C_N}(\alpha_N)$
can be written as ${1}/{\alpha_N}+o\left(B_0(\alpha_N)-{1}/{\alpha_N})\right)$.
More informally, the term $(1-u)^{N\cdot}$ under the integral for these two expressions
reduces by an exponential factor their asymptotic behavior. 

These various estimations give finally that 
\[
\lim_{N\to+\infty} \E\left(e^{-\alpha_N T_0^{X_N}}\right)=\lim_{N\to+\infty} 
\frac{B_{C_N}(\alpha_N)}{B_0(\alpha_N)}
=\frac{1}{1+\alpha/\nu}.
\]
The proposition is proved. 
\end{proof}

\section{Critical Regime}\label{critsec}
In this section, it is assumed that $C_N\sim \nu N$, if $X_N(0)=0$, the fluid limit of the
process is given by $\nu(1-\exp(-t))$, the fluid boundary $\nu$ is reach at time
$t=+\infty$. In fact, with a second order description, the process $X_N(t)$ can be written
as  $X_N(t)\sim \nu(1-\exp(-t))N+Y(t)\sqrt{N}$ for some ergodic diffusion process $(Y(t))$, so that the hitting
time $T_{C_N}^{X_N}$ of the boundary is such that
\[
\exp\left(-T_{C_N}^{X_N}\right)\sim \frac{Y(T_{C_N}^{X_N})}{\nu\sqrt{N}},
\]
which gives  a rough  estimation $T_{C_N}^{X_N}\sim\log(\sqrt{N})$. The  following proposition
shows that this approximation is fact quite precise. See also Theorem~4 of Flajolet and
Huillet~\cite{Flajolet}. 
\begin{proposition}
If $C_N=\nu N+\delta\sqrt{N}+o(\sqrt{N})$ with $\nu<1$, $\delta\in\R$, and $X_N(0)=0$, then the sequence of random variables
\[
\left(T_{C_N}^{X_N}-{\log(N)}/{2}\right)
\]
converges in distribution to a random variable $Z$ on $\R$ whose Laplace transform at
$\alpha>0 $ is given by,
\begin{equation}\label{eqcrit}
\E\left(e^{-\alpha Z}\right)={\Gamma(\alpha)}\left/\int_0^{+\infty} \exp\left(u\frac{\delta}{\nu}
-\frac{u^2}{2}\frac{(1-\nu)}{\nu}\right)u^{\alpha-1}\,\diff u.\right.
\end{equation}
If $\delta=0$, then the variable $Z-\log \left({\nu}/{(1-\nu)}\right)/2$ has the following density on $\R$,
\[
x\mapsto \sqrt{{2}/{\pi}}\exp\left(-x-e^{-2x}/{2}\right). 
\]
\end{proposition}
\noindent
As usual $\Gamma$ refers to the Gamma function,
\[
\Gamma(\alpha)=\int_0^{+\infty}e^{-u}u^{\alpha-1}\,\diff u,\quad \alpha>0,
\]
see Whittaker and Watson~\cite{Whittaker:01} for example.
Note that the Laplace transform of the limit in distribution is the ratio of the Mellin 
transforms of the functions
\[
u\mapsto \exp(-u) \quad \text{ and } u\mapsto \exp\left(-u\frac{\delta}{\nu}
-\frac{u^2}{2}\frac{(1-\nu)}{\nu}\right).
\]
See Flajolet and Sedgewick~\cite{Flajolet-09} on Mellin transforms. 
\begin{proof}
Proposition~\ref{prop:laplace} gives the equation
\[
\E_0\left(e^{-\alpha T_{C_N}^{X_N}}\right) = 
\left.{\int_0^{1} \left(1-u\right)^N u^{\alpha-1}\diff u}\right/{\int_0^{1} \left(1-u\right)^{N-C_N} \left(1+\frac{\mu}{\nu }u\right)^{C_N} u^{\alpha-1}\diff u},
\]
for $\alpha>0$. The asymptotic behavior of the numerator is easy since
\[
\int_0^{1} \left(1-u\right)^N u^{\alpha-1}\diff u=\frac{1}{N^\alpha}\int_0^{N} \left(1-\frac{u}{N}\right)^N u^{\alpha-1}\diff u
\sim \frac{\Gamma(\alpha)}{N^\alpha}.
\]
The denominator can be expressed as
\[
\frac{1}{N^{\alpha/2}}\int_0^{\sqrt{N}} e^{f_N(u)}u^{\alpha-1}\diff u,
\]
with 
\begin{align*}
f_N(u)&=(N-C_N)\log\left(1-\frac{u}{\sqrt{N}}\right)+C_N\log\left(1+\frac{\mu}{\nu}\frac{u}{\sqrt{N}}\right)\\
&=-\frac{(1-\nu)}{\nu}\frac{u^2}{2}+\frac{\delta}{\nu}
u+o(1/\sqrt{N}).
\end{align*}
By using Lebesgue's Theorem, one gets therefore that, for $\alpha>0$,
\[
\lim_{N\to+\infty} \E_0\left(e^{-\alpha \left[T_{C_N}^{X_N}-\log N/2\right]}\right) = {\Gamma(\alpha)}\left/
\int_0^{+\infty}\exp\left(-\frac{(1-\nu)}{\nu}\frac{u^2}{2}+\frac{\delta}{\nu} u\right)u^{\alpha-1}\,\diff
  u\right.,
\]
and hence the first part of the proposition.  

Assume that $\delta=0$, a change of
variable gives
\[
\int_0^{+\infty} \exp\left(-\frac{u^2}{2}\frac{(1-\nu)}{\nu}\right)u^{\alpha-1}\,\diff u= 
\frac{1}{2}\left(\frac{2\nu}{1-\nu}\right)^{\alpha/2}\Gamma(\alpha/2).
\]
The Laplace transform of $Z$ can therefore be expressed as 
\[
\E\left(e^{-\alpha Z}\right)=2\left(\frac{1-\nu}{2\nu}\right)^{\alpha/2} \frac{\Gamma(\alpha)}{\Gamma(\alpha/2)}
=\left(\frac{1-\nu}{\nu}\right)^{\alpha/2} \frac{2^{\alpha/2}}{\sqrt{\pi}}\Gamma((\alpha+1)/2),
\]
by using Legendre's  duplication Formula for  Gamma functions. See Whittaker and
Watson~\cite{Whittaker:01} page~240. Since
\begin{multline*}
\frac{2^{\alpha/2}}{\sqrt{\pi}}\Gamma\left(\frac{\alpha+1}{2}\right)=
\frac{1}{\sqrt{\pi}}
\int_0^{+\infty} \exp\left({\alpha}\log(2u)/2 -{\log(u)}/{2}-u\right)\,\diff u\\
=\sqrt{\frac{2}{\pi}}\int_{-\infty}^{+\infty}e^{-\alpha
u}\exp\left(-u-e^{-2u}/{2}\right)\,\diff u,
\end{multline*}
with a change  of variables, one gets the  desired result on the distribution  of $Z$. The
proposition is proved.
\end{proof}

One concludes with  the hitting time of empty  state, one can remark that,  at the correct
time scale, the time  is half of the corresponding variable in  the sub-critical case. See
Proposition~\ref{prop:empty}. A simple, naive, explanation is as follows. For sub-critical
regime the process lives  in a region centered at $\nu N$ and whose  width is of the order
of $\sqrt{N}$  and therefore makes many  excursions in this region  before reaching $C_N$.
In  the critical  case  the process  lives near,  but  only on  one  side of  $\nu N$.  In
particular  it cannot  go  above $C_N$  and, consequently,  does  not waste  time on  such
excursions.
\begin{proposition}[Hitting Time of Empty State]
If $C_N=\nu N+o(\sqrt{N})$ with $\nu>0$ and $X_N(0)=C_N$, then the sequence of random variables
\[
\left(N(1-\nu)^NT_{0}^{X_N}\right)
\]
converges in distribution to an exponential random variable with parameter $2\nu$.
\end{proposition}
\begin{proof}
Denote $\alpha_N=N(1-\nu)^N$. Recall that
\[
\E_{C_N}\left(e^{-\alpha_N T_0^{X_N}}\right) = \frac{d_{N}(\alpha)B_{C_N}(\alpha)+b_{N}(\alpha)D_{C_N}(\alpha)}{d_{N}(\alpha)B_0(\alpha)+b_{N}(\alpha)D_0(\alpha)}.
\]
One starts with the asymptotic behavior  of $(d_N(\alpha_N))$,  by definition
\begin{multline*}
d_N(\alpha_N){=}\mu\int_0^1 (1-u)^{N-C_N-1}\left(1+\frac{\mu}{\nu}u\right)^{C_N}u^{\alpha_N}\,\diff u
\\= \mu \frac{1}{\sqrt{N}}\int_0^{\sqrt{N}} \left(1-\frac{u}{\sqrt{N}}\right)^{N-C_N-1}\left(1+\frac{\mu}{\nu}\frac{u}{\sqrt{N}}\right)^{C_N}\,\diff u+o\left(1/\sqrt{N}\right).\end{multline*}
Since $\eta=\nu$ and $\mu=1-\nu$, the relation $\nu(1-\eta)=\mu\eta$ holds and therefore
\begin{align*}
d_N(\alpha_N) &= (1-\nu)\frac{1}{\sqrt{N}}\int_0^{+\infty} \exp\left(-\frac{(1-\nu)}{\nu}\frac{u^2}{2}\right)\,\diff
 u+o\left(1/\sqrt{N}\right)\\
&=  \frac{1}{\sqrt{N}}\sqrt{\frac{\pi}{2}}\sqrt{\nu(1-\nu)}+o\left(1/\sqrt{N}\right).
\end{align*}
Note that, up to a term $-1$ in an exponent which does not play a role in the limiting behavior,
the quantity $b_N(\alpha)$  is almost $d_N(\alpha)$ with $\nu$ replaced by
$(1-\nu)$. Consequently $b_N(\alpha)$ has the same asymptotic expansion as $d_N(\alpha)$. 

The asymptotic behaviors of the quantities $B_0(\alpha_N)$, $D_0(\alpha_N)$,
$B_{C_N}(\alpha_N)$ and $D_{C_N}(\alpha_N)$ are the same as the ones obtained in the proof
of Proposition~\ref{prop:empty}. By gathering these various estimations one gets that
\[
\lim_{N\to+\infty} \E\left(\exp\left(-\alpha_N T_0^{X_N}\right)\right)={2}/{(2+\alpha/\nu)},
\]
the proposition is proved. 
\end{proof}


\providecommand{\bysame}{\leavevmode\hbox to3em{\hrulefill}\thinspace}
\providecommand{\MR}{\relax\ifhmode\unskip\space\fi MR }
\providecommand{\MRhref}[2]{%
  \href{http://www.ams.org/mathscinet-getitem?mr=#1}{#2}
}
\providecommand{\href}[2]{#2}


\begin{thebibliography}{10}

\bibitem{Bingham}
N.~H. Bingham, \emph{Fluctuation theory for the {E}hrenfest urn}, Advances in
  Applied Probability \textbf{23} (1991), no.~3, 598--611.

\bibitem{Crescenzo}
Antonio Di~Crescenzo, \emph{First-passage-time densities and avoiding
  probabilities for birth-and-death processes with symmetric sample paths}, J.
  Appl. Probab. \textbf{35} (1998), no.~2, 383--394.

\bibitem{Diaconis}
Persi Diaconis, R.~L. Graham, and J.~A. Morrison, \emph{Asymptotic analysis of
  a random walk on a hypercube with many dimensions}, Random Structures
  Algorithms \textbf{1} (1990), no.~1, 51--72.

\bibitem{Engset}
T.O. Engset, \emph{On the calculation of switches in an automatic telephone
  system}, Telektronikk \textbf{94} (1998), no.~2, 99--142, Translation of a
  {G}erman version (1918).

\bibitem{Feinsilver}
Philip~J. Feinsilver, \emph{Special functions, probability semigroups, and
  {H}amiltonian flows}, Lecture Notes in Mathematics, vol. 696, Springer,
  Berlin, 1978.

\bibitem{Feuillet:02}
Mathieu Feuillet and Philippe Robert, \emph{The time scales of a stochastic
  network with failures}, preprint. In Preparation, 2011.

\bibitem{Flajolet}
Philippe Flajolet and Thierry Huillet, \emph{Analytic combinatorics of the
  {M}abinogion urn}, Proceedings of Fifth Colloquium on Mathematics and
  Computer Science (Blaubeuren) (U.~R\"osler, ed.), September 2008, Discrete
  Mathematics and Theoretical Computer Science (DMTCS).

\bibitem{Flajolet-09}
Philippe Flajolet and Robert Sedgewick, \emph{Analytic combinatorics},
  Cambridge University Press, 2009.

\bibitem{Flegg}
Mark Flegg, Dmitri Gramotnev, and P~Pollett, \emph{Ehrenfest model for
  condensation and evaporation processes in degrading aggregates with multiple
  bonds}, Physical Review E \textbf{78} (2008), 031117--1--031117--9.

\bibitem{Fricker:07}
Christine Fricker, Philippe Robert, and Danielle Tibi, \emph{On the rates of
  convergence of {E}rlang's model}, Journal of Applied Probability \textbf{36}
  (1999), no.~4, 1167--1184.

\bibitem{Karlin:02}
Samuel Karlin and James McGregor, \emph{The classification of birth and death
  processes}, Transactions of the AMS \textbf{86} (1957), 366--400.

\bibitem{Karlin-2}
\bysame, \emph{Coincidence properties of birth and death processes}, Pacific
  Journal of Mathematics \textbf{9} (1959), 1109--1140.

\bibitem{Karlin}
\bysame, \emph{Ehrenfest urn models}, Journal of Applied Probability \textbf{2}
  (1965), 352--376.

\bibitem{Keilson:01}
J.~Keilson, \emph{{M}arkov chains models-rarity and exponentiality}, Applied
  Mathematical sciences, vol.~28, Springer Verlag, New York, 1979.

\bibitem{Kennedy:02}
D.P. Kennedy, \emph{Some martingales related to cumulative sum tests and
  single-server queues}, Stochastic Processes and their Applications \textbf{4}
  (1976), 261--269.

\bibitem{Lamperti:06}
J.~Lamperti and J.~L. Snell, \emph{Martin boundaries for certain {M}arkov
  chains}, Journal of the Mathematical Society of Japan \textbf{15} (1963),
  113--128.

\bibitem{Revuz:01}
Daniel Revuz and Marc Yor, \emph{Continuous martingales and {B}rownian motion},
  third ed., Grundlehren der Mathematischen Wissenschaften [Fundamental
  Principles of Mathematical Sciences], vol. 293, Springer-Verlag, Berlin,
  1999.

\bibitem{Robert:08}
Philippe Robert, \emph{Stochastic networks and queues}, Stochastic Modelling
  and Applied Probability Series, vol.~52, Springer, New-York, June 2003.

\bibitem{rogers-00}
L.~C.~G. Rogers and D.~Williams, \emph{Diffusions, {M}arkov processes \&
  martingales vol. 2: It\^o calculus}, Cambridge University Press, 2000 (1987).

\bibitem{Simatos}
Florian Simatos and Danielle Tibi, \emph{Spatial homogenization in a stochastic
  network with mobility}, The Annals of Applied Probability \textbf{20} (2010),
  no.~1, 312--355.

\bibitem{Whittaker:01}
E.~T. Whittaker and G.~N. Watson, \emph{A course of modern analysis}, Cambridge
  Mathematical Library, Cambridge University Press, Cambridge, 1996, Reprint of
  the fourth (1927) edition.

\end{thebibliography}
\end{document}